\documentclass[12pt]{article}
\usepackage[centertags]{amsmath}
\usepackage{amsfonts}
\usepackage{amssymb}
\usepackage{amsthm}
\input{amssym.def}

\newtheorem{Definition}{Definition}[section]
\newtheorem{Theorem}[Definition]{Theorem}
\newtheorem{Lemma}[Definition]{Lemma}

\newcommand{\lc}{\mathcal{L}}
\newcommand{\rc}{\mathcal{R}}
\newcommand{\hc}{\mathcal{H}}
\newcommand{\jc}{\mathcal{J}}

\title{\Large \bf Nil extensions of Clifford ordered semigroup}
\author{A. K. Bhuniya  and K. Hansda \\
\footnotesize{Department of Mathematics, Visva-Bharati
University,}\\
\footnotesize{Santiniketan-731235, West Bengal, India}\\
\footnotesize{anjankbhuniya@gmail.com}, \
\footnotesize{kalyanh4@gmail.com}}

\begin{document}

\maketitle

\begin{abstract}{\footnotesize}
In this paper we describe all those ordered semigroups which are the
nil extension of Clifford, left Clifford, group like, left group
like ordered semigroups.
\end{abstract}

\section{Introduction:}
Nil  extensions of a semigroup (without order), are precisely the
ideal extensions  by a nil semigroup. In 1984, S.
Bogdanovi$\acute{c}$ and S. Mili$\acute{c}$ \cite{bc1} characterized
the semigroups (without order) which are nil extensions of
completely simple semigroups, where as, a similar work was done by
J. L Galbiati and M.L Veronesi \cite{GV} in 1980. Authors like S.
Bogdanovi$\acute{c}$, M. Ciri$\acute{c}$, Beograd have investigated
this type extensions for regular semigroup, group, periodic
semigroup as well as completely regular semigroup(see \cite{bc2},
\cite{bco2}).

The notion of ideal extensions in ordered semigroups is actually
introduced  by N. Kehayopulu and M. Tsingelis \cite{ke2003}. In
\cite{Ke2009}, they have worked on ordered semigroups which are nil
extensions of Archimedean ordered semigroups. The concepts of nil
extensions have been extended to ordered semigroups by Y. Cao
\cite{Cao 2000},  with characterizing ordered semigroups which are
nil extensions of simple, left simple, t-simple  ordered semigroups.
Further he described complete semilattices of nil extensions of such
ordered semigroups.

The aim of this paper is to describe all those ordered semigroups
which are  the nil extension of Clifford, left Clifford,  group
like, left group like ordered semigroups. This research originates
from the  research  papers \cite{bc2}, \cite{bco2}.

Our paper organized as follows. The basic definitions and properties
of ordered semigroups  are presented in Section 2. Section 3 is
devoted to characterizing the nil extensions of Clifford and left
Clifford ordered semigroups.

\section{Preliminaries:}
In this paper $\mathbb{N}$ will provide the set of all natural
numbers. An ordered semigroup is a partiality ordered set $S$, and
at the same time a semigroup $(S,\cdot)$ such that $(
\textrm{forall} \;a , \;b , \;x \in S ) \;a \leq b \Rightarrow
\;xa\leq xb \;\textrm{and} \;a x \leq b x$. It is denoted by
$(S,\cdot, \leq)$. For an ordered semigroup $S$ and $H \subseteq S$,
denote $(H] := \{t \in S : t \leq  h, \;\textrm{for  some} \;h\in
H\}$. An element $z$ in $S$ is called zero of $S$ if $z < x$ and
$zx= xz=z$ for every $x \in S$.

Let $I$ be a nonempty subset of an ordered semigroup $S$. $I$ is a
left (right) ideal of $S$, if $SI \subseteq I \;( I S \subseteq I)$
and $(I]= I$. $I$ is an ideal of $S$ if $I$ is both a left and a
right ideal of $S$. An (left, right) ideal $I \;of S$ is proper if
$I \neq S$. The intersection of all ideals of an ordered semigroup
$S$, if nonempty, is called the kernel of $S$ and is denoted by
$K(S)$.

An ordered semigroup $S$ is called a group like ordered semigroup if
for all $a, b \in S \;\textrm{there are} \;x, y \in S \;\textrm{such
\;that} \;a \leq xb \;\textrm{and} \;a \leq by$ \cite{bh1} and  $S$
left group like  if for all $a, b \in S \;\textrm{there is} \;x \in
S$ such that $a \leq xb$ \cite{bh1}. Kehayopulu \cite{Ke2006}
defined Greens relations on an ordered semigroup $S$ as follows: $ a
\lc b \; if   \;L(a)= L(b),
  \;a \rc b   \; \textrm{if}   \;R(a)= R(b),  \;a \jc b   \; \textrm{if}   \;I(a)= I(b), \;\textrm{and} \;\hc= \;\lc \cap \;\rc$.
These four relations $\lc, \;\rc, \;\jc \;\textrm{and} \;\hc$ are
equivalence relations. In an  ordered semigroup $S$, an equivalence
relation $\rho$ is called left (right) congruence if for $a, b, c
\in S \;a \;\rho \;b \; \textrm{implies} \;ca \;\rho \;cb \;(ac
\;\rho \;bc)$. $\rho$ is congruence if it is both left and right
congruence. A congruence $\rho$ on $S$ is called semilattice if  for
all $a, b \in S \;a \;\rho \;a^{2} \;\textrm{and} \;ab \;\rho \;ba$.
A semilattice congruence $\rho$ on $S$ is called complete if $a \leq
b$ implies $(a, ab)\in \rho$. The ordered semigroup $S$  is called
complete semilattice of subsemigroup of type $\tau$ if there exists
a complete semilattice congruence $\rho $ such that $(x)_{\rho}$ is
a type $\tau$ subsemigroup of $S$. Equivalently: There exists a
semilattice $Y$ and a family of subsemigroups $\{S\}_{\alpha \in Y}$
of type $\tau$ of $S$ such that:
\begin{enumerate}
\item \vspace{-.4cm}
$S_{\alpha}\cap S_{\beta}= \;\phi$ for any $\alpha, \;\beta \in \;Y
\;\textrm{with} \; \alpha \neq \beta,$
\item \vspace{-.4cm}
$S=\bigcup _{\alpha \;\in \;Y} \;S_{\alpha},$
\item \vspace{-.4cm}
$S_{\alpha}S_{\beta} \;\subseteq \;S_{\alpha \;\beta}$ for any
$\alpha, \;\beta \in \;Y,$
\item \vspace{-.4cm}
$S_{\beta}\cap (S_{\alpha}]\neq \phi$ implies $\beta \;\preceq
\;\alpha,$ where $\preceq$ is the order of the semilattice $Y$
defined by \\$\preceq:=\{(\alpha,\;\beta)\;\mid
\;\alpha=\alpha\;\beta\;(\beta\;\alpha)\}$ \cite{ke1}.
\end{enumerate}

Let $S$ be an ordered semigroup with the zero $0$.  An element $a
\in S$ is called a nilpotent if $a^n= 0$ for some $n \in
\mathbb{N}$. The set of all nilpotents of $S$ is denoted by
$Nil(S)$. $S$ is called nil ordered semigroup (nilpotent) \cite{Cao
2000} if $S = Nil(S)$.

Due to Cao \cite{Cao 2000} the definition of nil extension of
ordered semigroup is as follows.
\begin{Definition}\cite{Cao 2000}
Let $I$ be an ideal of an ordered semigroup $S$. Then $(S/I, \cdot,
\preceq)$ is called the Rees factor ordered semigroup of $S \;modulo
\;I$, and $S$ is called an ideal extension of $I$ by ordered
semigroup $S/I$. An ideal extension $S \;of \;I$ is called a
nil-extension of $I$ if $S/I$ is a nil ordered semigroup.
\end{Definition}

\begin{Lemma}\cite{Cao 2000}
 Let $S$ be an ordered semigroup and $I$ an ideal of $S$. Then
the following are equivalent:
\begin{itemize}
\item[(i)] $S$ is a nil-extension of $I$;
\item[(ii)] $(\forall a\in S)(\exists m \in \mathbb{N}) \;a^m \in I$.
\end{itemize}
\end{Lemma}
In {\cite{bh1}, we have  introduced the notion of Clifford and left
Clifford ordered semigroups and characterized their structural
representation. For the correspondences  of the results of this
paper we are  stating some of them.

\begin{Definition}\cite{bh1}
A regular ordered semigroup $S$ is called  Clifford ordered (left
Clifford ordered ) semigroup if for all $a, b  \in S$ and $ab
\in(bSa] \;(ab \in (Sa])$.
\end{Definition}
\begin{Theorem}\cite{bh1}\label{Clifford}
Let $S$ be a regular ordered semigroup. Then followings hold in $S$:
\begin{enumerate}
 \item \vspace{-.4cm}
 $S$ is Clifford if and only if $\lc = \rc$.
  \item \vspace{-.4cm}
  $\lc$ is a complete semilattice congruence if $S$ is Clifford.
   \item \vspace{-.4cm}
$S$ is Clifford ordered semigroup if and only if it is a complete
semilattice of  group like ordered semigroups.
\end{enumerate}
\end{Theorem}

\begin{Theorem}\cite{bh1}\label{left Clifford}
Let $S$ be a regular ordered semigroup. Then followings hold in $S$:
\begin{enumerate}
 \item \vspace{-.4cm}
  $\lc$ is a complete semilattice congruence if $S$ is left  Clifford.
   \item \vspace{-.4cm}
$S$ is  left Clifford ordered semigroup if and only if it is a
complete semilattice of left group like ordered semigroups.
\end{enumerate}
\end{Theorem}

\section{Main Results:}
Now we describe all those ordered semigroups which are the nil
extensions of Clifford, left Clifford, group like, left group like
ordered semigroups. We omit the proof of the following lemma as it
is straightforward.
\begin{Lemma}\label{LGO}
An ordered semigroup $S$ is  left group like ordered semigroup if
and only if $ a \in(aSab]$ if for all $a, \;b \in S$.
\end{Lemma}
\begin{Theorem}\label{ne6}
An ordered semigroup $S$ is a nil extension of a left group like
ordered semigroup if and only if  for every  $a,b \in S,  \;there
\;exists \;n \in \mathbb{N} \;such \;that \;a^{n} \in (a^{n}Sa^{n}b]
$ and for every  $ a\in S, \;b\in Reg_\leq(S), \;a\leq ba$ implies
$a\leq axb \;for \;some \;x\in S $.
\end{Theorem}
\begin{proof}
Suppose that $S$ is a nil extension of a left group like ordered
semigroup $K$ and $a,b \in S$. Then there is $m \in \mathbb{N}$ such
that $a^{m} \in K$.  Regularity of $ K$ implies  that $a^{m}\leq
a^{m}xa^{m}$ for some $x \in K$. Further, for $xa^{m},a^{m}b \in K$;
the left simplicity of $K$ yields that $xa^{m} \leq ya^{m}b,
\;\textrm{for some} \;y \in K$.  This  gives that $a^m \leq
a^{m}xa^{m} \leq a^{m}ya^{m}b$.  Next let $b \in Reg_\leq(S)$ and
$a\in S$ such that $a\leq ba$. Since $b \in Reg_\leq(S),
\;\textrm{there exists} \;z \in S$ such that $b\leq b(zb)^{n}
\;\textrm{for all} \;n \in \mathbb{N}$. Then for some $n_{1}\in
\mathbb{N}$,  $(zb)^{n_{1}} \in K$; whence $b(zb)^{n_{1}} \in K$.
Thus $b\in K$ and so  $ba\in K$. Thus  $a\in K$. Since $K$ is a left
group like ordered semigroup, for $a,ba\in K$ it follows that $a\leq
asba$ for some $s \in K$, by Lemma \ref{LGO}. Thus the given
conditions follow.

Conversely,  assume that given conditions hold in $S$. Let $a\in S$
be arbitrary. Then by  given condition we have $a^{m} \leq
a^{m}xa^{m+1}$, for some $x\in S$ and $m\in\mathbb{N}$. This implies
$a^{m+1} \in Reg_\leq(S)$ and so $Reg_\leq(S)\neq \phi$. Denote $T=
Reg_\leq(S)$.

Let us choose  $s\in S \;\textrm{and} \;a\in T$. Then the definition
of $T$ implies
$$a \leq a(xa)^{n}, \;\textrm{for all} \;n \in \mathbb{N}
\;\textrm{and some }  \;x \in S.$$ Thus  $sa \leq sa(xa)^{n},
\;\textrm{for all} \;n\in \mathbb{N}$.  Now  for $xa, sa \in S,
\;\textrm{there exists} \;m_{1} \in \mathbb{N}$ and $t_{1}\in S$
such that
$$(xa)^{m_{1}}\leq (xa)^{m_{1}}t_{1}(xa)^{m_{1}}sa, \;\textrm{by the first condition}.$$
Then $a \leq a(xa)^{m_1}$ implies $sa \leq sa(xa)^{m_1}$ and hence
$sa\leq sa(xa)^{m_{1}}t_{1}(xa)^{m_{1}}sa$; where
$(xa)^{m_{1}}t_{1}(xa)^{m_{1}}\in S$. So $sa \in T$. Also   $a \leq
(ax)^{n}a \;\textrm{for all} \;n\in\mathbb{N}$. Let $m_{2}\in
\mathbb{N}$ be such that $(ax)^{m_{2}} \in T$. Then  $as \leq
(ax)^{m_2}as$ implies by the  second  condition  that
\begin{equation}\label{eq5.1}
as \leq ast_{2}(ax)^{m_{2}} \;\textrm{for some} \;t_{2}\in S.
\end{equation}
Denote  $t= (ax)^{m_{2}}$.  Then the definition of $T$ implies
\begin{equation}\label{eq5.2}
t\leq t(zt)^{n} \;\textrm{for all} \;n \in \mathbb{N} \;\textrm{for
some} \;\;z\in T.
\end{equation}
Then  using the first  condition for  $zt, as \in S$ we have that
$(zt)^{m_{3}} \leq (zt)^{m_{3}}t_{3}(zt)^{m_{3}}as\\ \;\textrm{for
some} \;m_{3} \in \mathbb{N} \;\textrm{and} \;t_3 \in S$. That is $t
\leq t(zt)^{m_{3}}t_{3}(zt)^{m_{3}}as$, by   (\ref{eq5.2}). So from
(\ref{eq5.1}) we have $as \leq
ast_2t(zt)^{m_{3}}t_{3}(zt)^{m_{3}}as$, and hence  $as \in T$.

Next choose  $a\in S \;\textrm{and} \;b\in T$ such that $a\leq b$.
Since $b \in T$ there is $x \in S$ such that $b \leq b (xb)^n$ for
all $ n \in \mathbb{N}$. Now for $xb, a \in S$, it follows from the
first  condition that
\begin{equation}\label{eq5.3}
(xb)^{m_4}\leq (xb)^{m_4} t_4 (xb)^{m_4} a \;\textrm{for some }
\;m_4 \in \mathbb{N} \;\textrm{and}  \;t_4 \in S.
\end{equation}
So $a\leq b$ implies that
\begin{align*}
 a&\leq b(xb)^{m_4} t_4 (xb)^{m_4} a\\
  & = bt_5a; \;\textrm{where} \;t_5= (xb)^{m_4} t_4 (xb)^{m_4}.
\end{align*}
Since $b \in T$, by above we have $ \;bt_5 \in T$. Say $bt_5= t_6$.
Then for $a \in S, \;\textrm{and} \;t_6 \in T, \;a \leq t_6a$ yields
that $a \leq at_7t_6$ for some $t_7 \in S$, by second  condition.
Therefore
\begin{align*}
a &\leq at_7b t_5\\
  & = at_7b (xb)^{m_4} t_4 (xb)^{m_4}\\
  & \leq at_7b(xb)^{m_4} t_4 (xb)^{m_4} t_4 (xb)^{m_4}a, \;\textrm{by} \;(\ref{eq5.3}).
\end{align*}
Clearly $t_7b(xb)^{m_4} t_4 (xb)^{m_4} t_4 (xb)^{m_4} \in T$, as $b
\in T$. Thus $a \in T$, which  shows that $T$ is an ideal of $S$.

Finally let $c,d\in T$. Then there is $v \in S \;\textrm{such that}
\;c\leq c(vc)^{n} \;\textrm{for all} \;n \in \mathbb{N} $. Now for
$vc, d \in S, \;\textrm{there exists} \;t_{8} \in S$ such that
$c\leq c(vc)^{m_{5}}t_{8}(vc)^{m_{5}}d  \;\textrm{for some} \;m_{5}
\in \mathbb{N}$. Since $c \in T, \;c(vc)^{m_{5}}t_{8}(vc)^{m_{5}}
\in T$. Hence $T$ is  left simple. Thus $T$ is left group like
ordered semigroup such that for every $a \in S \;\textrm{there is}
\;m \in \mathbb{N}, \;a^{m}\in T$. Hence $S$ is nil extension of a
left group like ordered  semigroup $T$.
\end{proof}

In the following result we provide an independent proof of Corollary
5.2 of \cite{Cao 2000}.
\begin{Theorem}\label{ne7}
An ordered semigroup $S$ is a nil extension of a group like ordered
semigroup if and only if for all $a, b \in S, \;there \;exists \;n
\in \mathbb{N} \;such \;that \;a^{n}\in (b^{n}Sb^{n}]$.
\end{Theorem}
\begin{proof}
Suppose that  $S$ is a nil extension of a group like ordered
semigroup $G$ and $a,b\in S$. Then there exists $n\in \mathbb{N}
\;\textrm{such that} \;a^{n}, b^{n} \in G$. Since $G$ is a group
like ordered semigroup, there exists $u \in G \;\textrm{such that}
\;a^{n}\leq b^{n}u$. Also for $u, b^{n} \in G \;\textrm{there
exists} \;x \in G \;\textrm{such that} \;u\leq xb^n$. This implies
$b^{n}u \leq b^{n}xb^{n}$. Thus $a^{n} \leq b^{n}xb^{n}\;
\textrm{and hence} \;a^n \in (b^{n}Sb^{n}]$.

Conversely, let us assume that given condition holds in $S$. Choose
$a\in S$. Then for some $m \in\mathbb{N} \;\textrm{and} \;x\in S,
\;a^{m}\leq a^{m}xa^{m}$. Thus $Reg_\leq (S) \neq \phi$. Say
$G=Reg_\leq (S)$. So for every $a \in S,\;\textrm{there exists}
\;m\in \mathbb{N}$ such that $a^{m} \in G$. Let us consider  $b \in
G \;\textrm{and} \;s \in S$. Then  for all $n\in \mathbb{N}$
\begin{equation}\label{eq5.4}
bs \leq bybs \leq (by)^{n}bs \;\textrm{for} \;y\in S.
\end{equation}
Using the given condition for $bs, by\in S$, we  obtain $(by)^{m}bs
\leq (bs)^{m}z (bs)^{m+1} \;\textrm{for} \;z \in S \;\textrm{and}
\;m \in \mathbb{N}$.  This yields that
 \begin{align*}
 bs &\leq (bs)^{m}z(bs)^{m+1}\\
    & \leq bstbs; \;\textrm{where} \;t=bs^{m-1}zbs^{m}.
\end{align*}
Thus  $bs \in G$. Similarly $sb \in G$.

Next let $a \in S$ and $b\in G$ be such that $a \leq b$. Since $b
\in G$ there exists $x \in S $ such that $b \leq bxb$ and hence $b
\leq (bx)^n b (xb)^n$ for all $n \in \mathbb{N}$, which implies that
\begin{align*}
a &\leq a^m (z_1a^mba^nz)a^n, \;\textrm{for some} \;m,n
\in\mathbb{N} \;\textrm{and}
\;z, z_1\in S\\
  &=ata, \;\textrm{where} \;t= a^{m-1} z_1a^mba^nza^{n-1}.
\end{align*}
So $a\in G$. Hence $G$ is an ideal of $S$.

Finally,  consider  $a,b \in G$. Then   there exists $x \in S$ such
that
$$a \leq (ax)^{n}a \;\textrm{for all} \;n \in \mathbb{N},$$ and so
by the   given condition it follows that $a \leq b^{m'}z'b^{m'}a$
for some $m' \in \mathbb{N} \;\textrm{and} \;z' \in S$. This gives
that $a \leq bu \;\textrm{for  some} \;u= b^{m'-1}z'b^{m'} a \in G$.
Similarly there is some $v \in S$ such that $a \leq vb$. This shows
that $G$ is a group like ordered semigroup. Hence $S$ is a nil
extension of a group like ordered semigroup $G$.
\end{proof}
\begin{Theorem}\label{ne8}
An ordered semigroup $S$ is a nil extension of a Clifford ordered
semigroup if and only if  for every  $x, a, y \in S, \;there
\;exists \;n \in \mathbb{N} \;such \;that \;xa^{n}y \in
(xa^{n}ySya^{n}x]\cap (ya^{n}xSxa^{n}y]$ and  $a\in S, \;b\in
Reg_\leq(S)$ such that $a\leq b$,  implies $a\in (Sab]$.
\end{Theorem}
\begin{proof}
First suppose that  $S$ is   a nil extension of a Clifford ordered
semigroup $K$. Let $x,a,y \in S$. Then there is  $\;m \in
\mathbb{N}$ such that $a^{m}\in K$. Since $K$ is an ideal of $S$,
$xa^{m}y \in K$. Since  $K$ is a regular,  there exists $z_{1} \in
K$ such that
\begin{equation}\label{eq5.7}
xa^{m}y \leq xa^{m}yz_{1}xa^{m}y.
\end{equation}
Now   $z_{1}x,a^{m}y \in K$ implies that
\begin{equation}\label{eq5.8}
 z_{1}xa^{m}y \leq (a^{m}y) u_{1} (z_{1}x), \;\textrm{for some}
\;u_{1}\in K, \;\textrm{since} \;S \;\textrm{Clifford}.
\end{equation}
Similarly for $a^{m}, (yu_{1}z_{1}) \in K $ there is $u_2 \in S$
such that
\begin{equation}\label{eq5.9}
a^{m}(yu_{1}z_{1})\leq (yu_{1}z_{1})u_{2}a^{m}.
\end{equation}
Therefore
\begin{align*}
xa^{m}y &\leq xa^{m}yz_{1}xa^{m}y\\
        &\leq xa^my a^m (yu_1z_1)x, \;\textrm{by (\ref{eq5.8})}\\
        & \leq xa^{m}y^2u_{1}z_{1}u_{2}a^{m}x, \;\textrm{by
        (\ref{eq5.9})}.
\end{align*}
Thus
\begin{equation}\label{eq5.10}
xa^my \leq xa^{m}yz_{1}xa^{m}y^2u_{1}z_{1}u_{2}a^{m}x.
\end{equation}
 Also, for
$a^{m}y^{2}, u_{1} z_{1}u_{2}\in K \;\textrm{there exists}
\;u_{3}\in K $ such that $a^{m}y^2u_{1}z_{1}u_{2} \leq
u_{1}z_{1}u_{2}u_{3}a^{m}y^2$.  Then from (\ref{eq5.10}), we obtain
that
\begin{align*}
xa^{m}y  &\leq xa^{m}y(z_{1}xu_{1}z_{1}u_{2}u_{3}a^{m}y)ya^{m}x\\
          & \leq xa^{m}y s ya^{m}x; \;\textrm{where}  \;s= z_{1}xu_{1}z_{1}u_{2}u_{3}a^{m}y.
\end{align*}
Therefore $xa^{m}y \in (xa^{m}ySya^{m}x]$. Similarly  $xa^{m}y \in
(ya^{m}xSxa^{m}y]$.

Now $K \subseteq Reg_\leq(S)$ implies that  $Reg_\leq(S)\neq \phi$.
Consider $b\in Reg_\leq(S) \;\textrm{and} \;a\in S \;\textrm{such
that} \;a \leq b$. Since $b \in Reg_\leq(S), \;\textrm{there exists}
\;z \in S$ such that $b\leq (bz)^{n}b \;\textrm{for all} \;n\in
\mathbb{N}$. Since $S$ is a  nil extension of $K, \;\textrm{there
exists} \;n_{1}\in \mathbb{N}$ such that $(bz)^{n_{1}} \in K$. This
gives $(bz)^{n_{1}}b \in K$, which gives $b\in K$ and so $a\in K$,
since $K$ is an ideal of $S$. Since $K$ is a Clifford ordered
semigroup, by Theorem \ref{Clifford}(ii) $\lc$ is a congruence on
$S$. Since $a, b \in K$ we have $ \;a \lc ab$ and hence $a\subseteq
(Sab]$.

Conversely, let us assume that given conditions hold in $S$. Let
$a\in S$ be arbitrary. Then by the first  condition there exists $n
\in \mathbb{N}$ such that $a^{n+2}\leq a^{n+2}xa^{n+2},\\
\;\textrm{for some} \;x\in S$.  Thus $Reg_\leq(S)\neq \phi$. Say $T=
Reg_\leq(S)$. It is now clear that for each $a\in S, \;\textrm{there
exists } \; m\in \mathbb{N}$ such that $ a^m \in T$.

Let $s\in S$ and $x\in T$. Then  for all $n \in \mathbb{N}$ and for
some $\;t\in S, \;x  \leq (xt)^{n}x$ which  implies that $sx \leq
sx(tx)^{n-1}tx, \;\textrm{for all} \;n \in \mathbb{N}$. By first
condition there are  $ s_{1} \in S \;\textrm{and} \;m_{1} \in
\mathbb{N}$ such that $sx \leq
sx(tx)^{m_{1}}txs_{1}tx(tx)^{m_{1}}sx$ and thus $sx \leq sxpsx;
\;\textrm{where} \;p=(tx)^{m_{1}}txs_{1}tx(tx)^{m_{1}}$.  Also for
every  $n\in\mathbb{N}$,
\begin{align*}
xs &\leq x(tx)^{n}s\\
   & \leq xt(xt)^{n-1}xs.
\end{align*}
So there is  $m_2 \in \mathbb{N}$ such that
\begin{align*}
xs &\leq xs(xt)^{m_{2}}xts_{2}xt(xt)^{m_{2}}xs\\
   & \leq xsqxs; \;\textrm{where} \;q=(xt)^{m_{2}}xts_{2}xt(xt)^{m_{2}}.
\end{align*}
Thus $sx, sx \in T$.

To show  $T$, a Clifford ordered semigroup, choose   $a,b\in T$.
Then there is  $r\in S$ such that
\begin{align}\label{eq5.11}
ab &\leq abrab \nonumber\\
   & \leq (abra)(bra)^{n-1}b, \;\textrm{for all} \;n \in \mathbb{N}.
\end{align}
Now for  $abra,bra, b \in S$,  the  first condition yields that
$$abra(bra)^{m_{3}}b \leq b(bra)^{m_{3}}abrap_{1}abra(bra)^{m_{3}}b
\;\textrm{for some} \;p_{1}\in S \;\textrm{and} \;m_{3}\in
\mathbb{N}.$$ Therefore from (\ref{eq5.11})
\begin{align}\label{eq5.12}
ab &\leq b(bra)^{m_{3}}abrap_{1}abra(bra)^{m_{3}}b \nonumber \\
   & \leq b(bra)^{m_{3}}abrap_{1}abra(bra)^{m_{3}-1}brab \nonumber\\
   & \leq bgab; \;\textrm{where} \;g=(bra)^{m_{3}}abrap_{1}abra(bra)^{m_{3}-1}br\in T.
\end{align}
Similarly there are  $m_4 \in \mathbb{N}$ and $p_2 \in S$  such that
$ab \leq abra (bra)^{m_4} bp_2b (bra)^{m_4} abra$. So from
(\ref{eq5.12}),
$$ab \leq b(g abr a(bra)^{m_4}b p_2 b(bra)^{m_4}) abra= bua;
\;\textrm{where} \;u= g abr a(bra)^{m_4}b p_2 b(bra)^{m_4} abr \in
T.$$ Hence $T$ is Clifford ordered semigroup.

Now let $a\leq b$ for some $a\in S$ and $b \in T$. Then by the
second  condition, there is $z\in S$ such that  $a \leq zab,
\;\textrm{that is}$,
\begin{equation}\label{eq5.13}
a \leq zabtab \;\textrm{for some} \;t\in S.
\end{equation}
Since $T$ is Clifford ordered semigroup, for $zabt,ab\in T$ it
follows that $$zabtab\leq abp_{3}zabt \;\textrm{for some} \;p_3 \in
S.$$ Similarly  for $bp_{3}za, bt \in T$, we have  $bp_{3}zabt \leq
btp_{4}bp_{3}za \;\textrm{for some} \;p_4 \in S$.

The last two inequalities together with (\ref{eq5.13})  yields that
$a\leq aha, \;\textrm{where} \;h=btp_{4}bp_{3}z$. Thus $a \in T$ and
so $T$ is an ideal of $S$. Hence $S$ is a nil extension of a
Clifford ordered semigroup $T$.
\end{proof}
\begin{Theorem}\label{ne9}
An ordered semigroup $S$ is a nil extension of a left Clifford
ordered semigroup if and only if  for every $x, a, y \in S, \;there
\;exists \;n \in \mathbb{N} \;such \;that \;xa^{n}y \in
(xa^{n}ySya^{n}x]\cap (xa^{n}ySxa^{n}y]$ and  $a\in S, \;b\in
Reg_\leq(S)$ such that $a\leq b$ implies  $a\leq azab \;for \;some
\;z\in S$.
\end{Theorem}
\begin{proof}
Let $S$ be  a nil extension of a left Clifford ordered semigroup
$K$. Choose $x,a,y \in S$. Then there exists $\;m \in \mathbb{N}$
such that $a^m \in K$. Since $K$ is an ideal of $S$, $xa^my \in K$.
Also the regularity of $K$  yields that
\begin{equation}\label{eq5.14}
xa^m y \leq xa^my z_1 xa^my, \;\textrm{for some} z_1 \in K.
\end{equation}
Since $K$ is a left Clifford ordered semigroup and  $z_1x, a^my \in
K$, by Theorem \ref{left Clifford} it follows that  $z_1x a^my \leq
z_2(z_1x)$ for some $z_2 \in K$. Therefore
\begin{align}\label{eq5.15}
xa^m y &\leq xa^my z_1 xa^my \nonumber\\
       &\leq xa^my z_1 xa^my  z_1 xa^my \nonumber\\
       &\leq xa^my z_1 xa^my z_2(z_1x).
\end{align}
Similarly, for $a^m, yz_2z_1 \in K$ there is $z_3 \in K$ such that
\begin{equation}\label{eq5.16}
a^myz_2z_1 \leq z_3a^m,
\end{equation}
and for $a^m, z_1xz_3 \in K$ there is $z_4 \in K$ such that
\begin{equation}\label{eq5.17}
a^myz_1xz_3 \leq z_4 a^my.
\end{equation}
Thus from (\ref{eq5.15}) we obtain that
\begin{align*}
xa^my &\leq xa^my z_1 xa^my z_2z_1x \nonumber\\
      &\leq xa^my z_1 xa^m z_3 a^mx , \;\textrm{from (\ref{eq5.16}}) \nonumber\\
      &\leq xa^my z_4 a^mya^mx, \;\textrm{from (\ref{eq5.17})}.
\end{align*}
Hence $xa^{m}y \in (xa^{m}ySya^{m}x]$ and so $xa^{m}y \in
(xa^{m}ySya^{m}x]\cap (xa^{m}ySxa^{m}y]$, from (\ref{eq5.14}).

To show the second condition  choose $a \in S$ and $b \in
Reg_\leq(S)$ be such that $a \leq b$. By the regularity of $b$
yields that $b \leq b(tb)^n$ for some $t \in S$ and for all $n \in
\mathbb{N}$. Then there is $r \in \mathbb{N}$ such that  $(tb)^r \in
K$. Since $K$ ia an ideal  $b(tb)^r \in K$ and so $b \in K$. Thus $a
\in K$. Since $K$ is left Clifford ordered semigroup, $\lc$ is
congruence on $S$, by Theorem \ref{left Clifford}(i). Thus $a \lc
ab$. Then $a, ab $ are in Theorem \ref{LGO}, $a \leq azab$ for some
$z \in S$. This proves  the necessary condition.

Conversely, suppose  that given conditions hold  in $S$. Let $a\in
S$. Then by the first  condition  there exists $n\in\mathbb{N}$ such
that $a^{n+2}\leq a^{n+2}xa^{n+2}, \;\textrm{for some} \;x\in S$.
Thus $Reg_\leq (S)\neq \phi$. Say $T= Reg_\leq(S)$. Now for each
$a\in S, \;\textrm{there exists} \;m \in \mathbb{N}$ such that $a^m
\in T$. Let $s\in S$ and $x \in T$. Then for all $n\in \mathbb{N}
\;\textrm{and for some} \;t\in S, \;x \leq (xt)^{n}x$. This implies
 $sx \leq sx(tx)^{n-1}tx, \;\textrm{for all} \;n \in
\mathbb{N}$.

Then by the first  condition there are $s_{1} \in S \;\textrm{and}
\;m_{1} \in \mathbb{N}$ such that
\begin{align*}
sx &\leq sx(tx)^{m_{1}}txs_{1}tx(tx)^{m_{1}}sx\\
   &\leq sxpsx; \;\textrm{where} \;p=(tx)^{m_{1}}txs_{1}tx(tx)^{m_{1}}.
\end{align*}
Therefore $ sx  \in T$.

We now  show that $T$ is  a left Clifford ordered semigroup. For
this let us assume that $a,b\in T$. Then there is $t_1 \in S$ such
that
$$
ab \leq ab t_1 ab \leq a(bt_1a)^n bt_1 ab \;\textrm{for all} \;n \in
\mathbb{N}.$$ Then by first condition  there are $t_2 \in S
\;\textrm{and} \;m' \in \mathbb{N}$ such that
\begin{align*}
ab &\leq a (bt_1a)^{m'}bt_1 ab \\
   &\leq a (bt_1a)^{m'} bt_1ab t_2 bt_1ab(bt_1a)^{m'} a
\end{align*}
Therefore  $ab \leq t'_{1}a; \;\textrm{where} \;t'_{1}=  a
(bt_1a)^{m'} bt_1ab t_2 bt_1ab(bt_1a)^{m'}.$ Since  $a \in T$ we
have $t_1 ^{'} \in T$. Hence $T$ is left Clifford ordered semigroup.

Next let  $ a\in S \;\textrm{and} \;b \in T$ such that $a \leq b$.
Using second condition we have $a \leq azab$, for some $z \in S$.
Then for all $n \in \mathbb{N,}  \;a \leq a(za)^n b^n$.  So for some
$m'' \in \mathbb{N}, \;(za)^{m''} \in T$.  Since $T$ is a left
Clifford, $(za)^{m''} b^{m''} \in (Ta]$. So $a \in (aSa]$, that is
$a \in T$.

Finally to show $T$, an ideal of $S$ we need only to show that $xs
\in T$. The regularity of $x$ yields that  $xs \leq x (tx)^n s$ for
all $n \in \mathbb{N}$. Also for some $l \in \mathbb{N}$,  $x (tx)^l
s \in T$. Then $xs \leq x(tx)^l s$ implies that $xs \in T$ by above.
Thus $T$ is an ideal of $S$. Hence $S$ is a nil extension of a left
Clifford ordered semigroup $T$.
\end{proof}

\bibliographystyle{plain}

\end{document}